\documentclass{amsart}

\usepackage{wwwlnbcommon}
\usepackage{wwwdb}
\usepackage[geom,slash,sets,norm,errorterm,pde,miscops,labeled,conv,roman]{www_math_commands}

\usepackage{amsrefs}
\usepackage{hyperref}
\usepackage{cleveref}

\newcommand\dvol{\mathrm{dvol}}
\newcommand\Div{\operatorname{div}}

\counterwithout{equation}{section}

\begin{document}
\title{Zeroes of Eigenfunctions of Schr\"odinger Operators after Schwartzman}
\author{Willie WY Wong}
\address{Department of Mathematics, Michigan State University, East Lansing, Michigan, United States}
\email{wongwwy@math.msu.edu}

\begin{abstract}
	Consider a complete, connected, smooth, oriented Riemannian manifold $(M,g)$ with boundary, such that the first Betti number vanishes. Sol Schwartzman proved that 
	for Schr\"odinger operators of the form $-\Delta_g + V$ where $\Im(V)$ is signed, if $f: M\to\Complex$ is a non-vanishing element of its kernel, then $f$ has constant phase. 
	The proof relied on dynamical systems methods applied to the gradient flow of the phase of $f$. 
	In this manuscript we provide a more direct PDE argument that proves strengthened versions of the same facts.  
\end{abstract}

\maketitle

Throughout let $(M,g)$ be a complete smooth oriented Riemannian manifold, possibly with boundary, and possibly non-compact. 
Write $\dvol_g$ for the metric volume form on $M$, and $\nabla$ for the Levi-Civita connection. 
The operator $\Div$ acting on a rank-1 tensor $\eta$ (either a one form or a vector field) is the metric contraction of $\nabla \eta$.
Denote by $\Delta_g = \Div\circ \nabla$ the (negative) Laplace--Beltrami operator for the metric $g$, and $V \in C^\infty(M;\Complex)$. 
We are interested in the kernel of the Schr\"odinger operator 
\begin{equation}
	L = - \Delta_g + V.
\end{equation}
Specifically, we are interested in necessary constraints for elements of the kernel that are nowhere vanishing on $M$. 
Our main results are that, firstly, when the imaginary part $\Im(V) \geq 0$ and is not identically zero, there do not exist nowhere vanishing elements in the kernel. Secondly, when $\Im(V)$ is identically zero, then under some additional assumptions we find nowhere-vanishing kernel elements must have constant phase. 

Our arguments will be based on an alternative version of a differential identity found by Schwartzman \cite{SS1}*{Lemma 1}. 
For arbitrary 
$f\in C^2(M;\Complex)$, if on some open set $U$ we have $f|_U \neq 0$, then $\frac{1}{f}$ is well-defined on $U$ and is also $C^2$. The following computation is valid on $U$:
\begin{equation}\label{eq:main:identity}
\begin{aligned}
	\mathrm{div}\Bigl( \bar{f}^2 \nabla \bigl( f/\bar{f} \bigr) \Bigr) & = \Div \bigl( \bar{f} \nabla f - f \nabla \bar{f}\bigr) \\
										   &= \bar{f} \Delta_g f - f \Delta_g \bar{f}  \\
										   &= \bar{f}(-L + V)f - f (-\bar{L} + \bar{V})\bar{f} = -2i \Im( \bar{f}\cdot Lf) + 2i \Im(V) \abs{f}^2. 
\end{aligned}
\end{equation}
Observe that in the case $f|_U = \exp \phi$ for some $\phi\in C^2(M;\Complex)$, the quantity
\[ \bar{f}^2 \nabla \bigl(f / \bar{f}\bigr) = e^{2\bar{\phi}} \nabla e^{\phi - \bar{\phi}} = 2 i \abs{f}^2 \nabla(\Im \phi). \]
Schwartzman stated this identity in terms of $\phi$; to ensure $\phi$ exists he assumed the underlying domain has vanishing first Betti number. 
Equation \eqref{eq:main:identity} provides a generalization. 

Since we allow $M$ to be possibly non-compact and possibly with boundary, we will need to restrict to functions with asymptotic / boundary conditions. Below a function $f$ is said to be \emph{admissible} if
\begin{itemize}
	\item $f\in C^2(M;\Complex)$
	\item \emph{(if $\partial M \neq\emptyset$)} $f$ satisfies the Neumann boundary conditions $\nu(\phi) = 0$ where $\nu$ is a normal vector field to the boundary $\partial M$. 
	\item \emph{(if $M$ is non-compact)} both $f$ and $\nabla f$ are square integrable on $M$ against $\dvol_g$. 
\end{itemize}

\begin{rmk}
	We consider the Neumann boundary conditions and not the Dirichlet one as the Dirichlet boundary condition is incompatible with $f$ being nowhere vanishing.
\end{rmk}

Let $\chi$ be a non-negative function of compact support on $M$. Then testing \eqref{eq:main:identity} against $\chi$ we find, for any non-vanishing $f$,
\[ 
	\int_{\partial M} \chi \bar{f}^2 \nu(f/\bar{f}) ~\dvol_{\partial M} - \int_M \bar{f}^2 g(\nabla\chi,\nabla(f/\bar{f})) ~\dvol_g = 2i \int_{M} \Im( - \bar{f} \cdot L f + V\abs{f}^2) \chi ~\dvol_g. 
\]
If we further assume that $f\in \ker(L)$ and is admissible, then this simplifies to 
\begin{equation}
	- \int_M g(\nabla \chi, \bar{f}\nabla f - f \nabla \bar{f}) ~\dvol_g = 2i\int_M \Im(V) \abs{f}^2 \chi ~\dvol_g. 
\end{equation}
Now the standard argument (here we use completeness of $g$) of taking a sequence $\chi_n$ converging (uniformly on compact sets) to the identity with $\norm[\infty]{\nabla \chi_n}\to 0$ shows (using that $\bar{f}\nabla f \in L^1$) that 
\begin{equation}\label{eq:limit}
	\lim_{n\to\infty} \int_M \chi_n \Im(V) |f|^2 ~\dvol_g = 0.
\end{equation}

\begin{thm}
	Assume $\Im(V)$ is non-negative (non-positive) and not identically 0, then any admissible $f\in \ker(L)$ must vanish somewhere. 
\end{thm}
\begin{proof}
	Suppose not, then $f$ is an admissible element of $\ker(L)$ that is nowhere vanishing. But then \eqref{eq:limit} applies. Since $\Im(V)$ is signed, we can conclude that $\Im(V) |f|^2 \equiv 0$. As $\Im(V)$ is assumed not identically zero, we must have $|f|$ vanishing somewhere, providing a contradiction. 
\end{proof}

\begin{rmk}
	This theorem is stronger than what was recorded by Schwartzman \cite{SS2}*{Theorem 1}. Schwartzman's result only applies to closed manifolds with vanishing first Betti number. Ours apply without any Betti number assumption, and to manifolds that are possibly non-compact and with boundary, without a priori bounds required on $V$. Furthermore, our conclusion rules out completely the possibility of non-vanishing $f$. 
\end{rmk}

We next consider the case where $\Im(V)\equiv 0$, where it was shown under the Betti number assumption any non-vanishing element of $f\in \ker(L)$ must have constant phase \cite{SS1}*{Theorem 2}. 
Before giving our simplified proof, we first observe that, unlike the previous theorem, some topological assumption is necessary. 

\begin{exa}
	Let $(M,g)$ be the unit circle with the standard metric, and let $V \equiv -1$. Then there exists a nowhere-vanishing solution to $L(f) = 0$ with non-constant phase, namely $f(\theta) = \exp i\theta$. 
\end{exa}

In lieu of placing a Betti number assumption on the manifold, we make the assumption that $f$ has a well-defined complex logarithm. (When $M$ has vanishing first Betti number, all non-vanishing $C(M;\Complex)$ functions admit well-defined continuous complex logarithms. As seen in the example above, this is no longer true for general $M$.)

\begin{thm}\label{thm:s1}
	Assume $V$ is real-valued. Suppose $f = \exp\phi$ for some $\phi\in C^2(M;\Complex)$, and that $f$ is admissible and $L(f) = 0$. Then $\Im(\phi)$ is locally constant. 
\end{thm}

\begin{proof}
	For convenience write $u = \Im(\phi)$. Note that $\nabla f = f (\nabla \Re\phi + i \nabla u)$, and hence $\abs{\nabla f}^2 = \abs{f}^2 \cdot ( \abs{\nabla \Re\phi}^2 + \abs{\nabla u}^2)$; admissibility of $f$ implies that $\abs{f} \nabla u\in L^2$. 
	Returning to \eqref{eq:main:identity}, we see that it now reads 
	\[ \Div(\abs{f}^2 \nabla u) = 0 .\]
	Multiply by $u$ and integrate by parts (using a standard approximation argument with compact cut-offs as above) we find 
	\[ 0 = \int_{M} \Div(\abs{f}^2 \nabla u)u ~\dvol_g = - \int_M \abs{f}^2 \abs{\nabla u}^2 ~\dvol_g. \]
	As $f$ is by assumption nowhere vanishing, we must have $\nabla u \equiv 0$. 
\end{proof}

\begin{rmk}
	Throughout we made no assumption on connectedness of $M$. If $M$ were assumed to be connected, Theorem \ref{thm:s1} can be upgraded to concluding a globally constant phase $\Im(\phi)$. 
\end{rmk}

\end{document}